\documentclass{amsart}

\usepackage{amsmath,amsthm, amssymb}
\usepackage[hideerrors]{xcolor}
\usepackage{pgf,tikz}
\usetikzlibrary{decorations.pathreplacing}
\usetikzlibrary{arrows.meta}
\usetikzlibrary{arrows}
\usetikzlibrary[patterns]
\usetikzlibrary{shapes,backgrounds}
\usepackage{graphicx}
\usepackage{hyperref}
\usepackage[capitalise, noabbrev]{cleveref}
	\crefname{subsection}{Subsection}{Subsections}
	\Crefname{subsection}{Subsection}{Subsections}
\usepackage{array}
\usepackage{rotating}
\usepackage{longtable}
\usepackage{epstopdf}
\usepackage[T1]{fontenc}
\usepackage{bold-extra}
\usepackage{mathrsfs}
\usepackage[foot]{amsaddr}

\newcommand{\snort}{\textsc{Snort }}
\newcommand{\col}{\textsc{Col }}

\newcommand{\domineering}{\textsc{Domineering }}
\newcommand{\Snort}{\textsc{Snort}}

\newcommand{\Domineering}{\textsc{Domineering}}

\DeclareMathOperator{\link}{link}
\renewcommand{\L}{\mathfrak{L}}
\newcommand{\R}{\mathfrak{R}}
\newcommand{\G}{\mathbb{G}}
\newcommand{\I}{\text{I}}
\newcommand{\V}{\mathbb{V}}
\DeclareMathOperator{\mex}{mex}
\DeclareMathOperator{\birth}{b}
\DeclareMathOperator{\birthformal}{\widetilde{b}}

\newcommand{\incomp}{\not\gtrless}

\theoremstyle{definition} \newtheorem{definition}{Definition}[section]
\theoremstyle{plain} \newtheorem{theorem}[definition]{Theorem}
\theoremstyle{plain} \newtheorem{corollary}[definition]{Corollary}
\theoremstyle{plain} \newtheorem{proposition}[definition]{Proposition}
\theoremstyle{plain} 
\theoremstyle{plain} 
\theoremstyle{definition} \newtheorem{example}[definition]{Example}
\theoremstyle{definition} \newtheorem{remark}[definition]{Remark}
\theoremstyle{definition} 
\theoremstyle{definition} 
\theoremstyle{definition} 

\usepackage{xspace}
\makeatletter
\DeclareRobustCommand\onedot{\futurelet\@let@token\@onedot}
\def\@onedot{\ifx\@let@token.\else.\null\fi\xspace}
 
\def\ie{{i.e}\onedot}

\makeatother

\newcommand{\Ldomino}[2]{\filldraw[fill=black!30!white] (#1+0.2,#2+0.2)--(#1+0.8,#2+0.2)--(#1+0.8,#2+1.8)--(#1+0.2,#2+1.8)--cycle;}
\newcommand{\Rdomino}[2]{\filldraw[fill=black!30!white] (#1+0.2,#2+0.2)--(#1+1.8,#2+0.2)--(#1+1.8,#2+0.8)--(#1+0.2,#2+0.8)--cycle;}

\begin{document}
\title{Game Values of Strong Placement Games}
\author{Svenja Huntemann \textsuperscript{1}}
\address{\textsuperscript{1} School of Mathematics and Statistics\\ Carleton University\\ Ottawa, Canada\\ and Dept.~of Mathematics and Statistics\\
Dalhousie University\\ Halifax, Canada}

\keywords{Combinatorial game, strong placement game, game value.}

\thanks{The author's research was supported by the Natural Sciences and Engineering Research Council of Canada (funding reference numbers PDF-516619-2018 and CGSD3-459150-2014) and the Killam Trust.}

\begin{abstract}
The legal positions of a strong placement games, such as \Domineering, form a simplicial complex called the legal complex. In this paper, we use the legal complex to study the game values taken on by the class of strong placement games using the legal complex. We show that many interesting values are possible, including all numbers and nimbers. We further consider how structures of the legal complex influence which values are possible.
\end{abstract}

\maketitle

\section{Introduction}

Combinatorial games are 2-player perfect information games such as Chess, Checkers, or Go. To each combinatorial game, or position in a game, one assigns a game value, which in some sense indicates which player has the advantage and by how much. A problem of interest in combinatorial game theory is the range of values that occur in a game. One of the most celebrated results in combinatorial game theory is the Sprague-Grundy Theorem which in essence states that the impartial game \textsc{Nim} takes on all game values possible for impartial games - those games in which both players always have the same options. Motivated by this is the search for a nontrivial finite game which takes on all games values possible, even for partizan games - those where the options for the players might differ.

A game taking on all possible values is called \textbf{universal}, and Carvalho and Santos \cite{CarvalhoS2019} recently constructed the first known nontrivial universal game. In this paper we study the possible values of strong placement (SP-)games, a class of games with nice structure. The game introduced in \cite{CarvalhoS2019} is not an SP-game, and it is still an open question whether the class of SP-games is universal.

The question of the range of values has received attention for some specific SP-games. The only complete result for partizan SP-games is that \col only takes on numbers and numbers plus star (see \cite[p.47]{BerlekampCG2004}). Some partial results are known for \textsc{Domineering} (see for example \cite{Kim1996,UiterwijkB2015}) and for \snort (see \cite[pp.181--183]{BerlekampCG2004}). 

Since SP-games are much easier to understand than many other combinatorial games, if we are able to show that SP-games take on all possible game values, the class of SP-games would provide an excellent new tool for studying combinatorial games. But even if SP-games are not universal, being able to restrict the possible values would simplify game value calculations.

In \cite{FaridiHN2019c} it was shown that each SP-game is in a one-to-one correspondence with a simplicial complex. We will take advantage of this fact in the exploration of the universality of SP-games in this paper. Although we are not able to determine their universality either way, we are able to show that many interesting values are possible.

In \cref{sec:background} we will give an introduction to SP-games and their connection with simplicial complexes, as well as all background  on game values required. Readers familiar with combinatorial game theory can skip the parts between \cref{def:basicPosition} and \cref{def:simplicialComplex} and between \cref{remark} and \cref{def:ValueSet}. In \cref{sec:ValueBirthday} we will concentrate on game values with small birthdays, essentially those games with few moves before the game ends. We then focus on specific game values in \cref{sec:values}, namely numbers, switches, and tinies. Finally, we discuss impartial SP-games in \cref{sec:impartial}.

Proofs for all combinatorial game theory results given without reference can be found in \cite{Siegel2013}.

\section{Strong Placement Games and Game Values Background}\label{sec:background}

A \textbf{combinatorial game} is a 2-player game with perfect information and no chance devices. We call the two players Left and Right. For our purposes a game will consist of a ruleset, which indicates what moves are available to either player, and a given starting position. A \textbf{position} is an arrangement of pieces on the board. Given a position $P$, if Left (Right) has a move from $P$ to $Q$, then $Q$ is called a Left (Right) \textbf{option} of $P$. The set of all Left options of $P$ is denoted by $P^{\mathcal{L}}$, while the Right options are $P^\mathcal{R}$. A game with starting position $G$ is written in shorthand as $G=\{G^\mathcal{L}\mid G^\mathcal{R}\}$.

A combinatorial game is called \textbf{short} if it only has finitely many positions and ends in a finite amount of time, \ie one cannot return to a previous position. The game ends if the player whose turn it is cannot make a move and under \textbf{normal play} this player looses. All games we consider are short and normal play.

\begin{definition}\label{def:SPgame}
A \textbf{strong placement game} (or \textbf{SP-game}) is a combinatorial game which satisfies the following condition:
\begin{itemize}
\item[(i)] The board is empty at the beginning of the game.
\item[(ii)] Players place pieces on empty vertices of the board according to the rules.
\item[(iii)] Pieces are not moved or removed once placed.
\item[(iv)] The rules are such that if it is possible to reach a position through a sequence of legal moves, then any sequence of moves leading to this position consists of legal moves.
\end{itemize}
\end{definition}

Note that condition (iv) in the above definition implies that the order of moves does not matter, and that the last piece played could have been played at any previous point. The property that positions are independent of the order of moves allows us to represent positions by faces of a simplicial complex.

Two well studied SP-games, which we will be using throughout, are the following rulesets with a suitable board.
\begin{definition}
In \textbf{\Snort} (see \cite{BerlekampCG2004}), players place a piece on a single vertex which is not adjacent to a vertex containing a piece from their opponent.

In \textbf{\Domineering} (see \cite{Berlekamp1988} or \cite{LachmannMR2002}), which is played on grids, both players place dominoes. Left may only place vertically, and Right only horizontally. The vertices of the board are the squares of the grid, and each piece occupies two vertices. 
\end{definition}

For SP-games, since the order of moves taken does not matter, the positions with a single piece played become very important. We thus define a basic position:
\begin{definition}\label{def:basicPosition}
A position with a single piece played, whether this is legal or not, is called a \textbf{basic position}.
\end{definition}

Any position in an SP-game is the union of a finite number of basic positions. For example, positions in \textsc{Snort} played on a path of three vertices break up as follows:

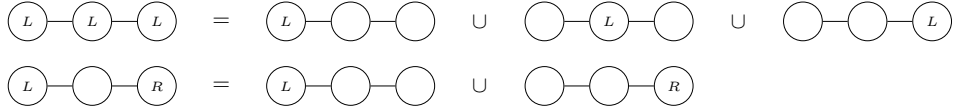
\begin{center}
\resizebox{\linewidth}{!}{
\begin{tikzpicture}[scale=1, vertex/.style={circle, draw, minimum size=6mm, font=\tiny}]
	\node[vertex] (1) at (0,0) {$L$};
	\node[vertex] (2) at (1,0) {};
	\node[vertex] (3) at (2,0) {$R$};
	\draw (1)--(2)--(3);
	\node at (3,0) {$=$};
	\node[vertex] (4) at (4,0) {$L$};
	\node[vertex] (5) at (5,0) {};
	\node[vertex] (6) at (6,0) {};
	\draw (4)--(5)--(6);
	\node at (7,0) {$\cup$};
	\node[vertex] (8) at (8,0) {};
	\node[vertex] (9) at (9,0) {};
	\node[vertex] (10) at (10,0) {$R$};
	\draw (8)--(9)--(10);
	
	\node[vertex] (1a) at (0,1) {$L$};
	\node[vertex] (2a) at (1,1) {$L$};
	\node[vertex] (3a) at (2,1) {$L$};	
	\draw (1a)--(2a)--(3a);
	\node at (3,1) {$=$};
	\node[vertex] (4a) at (4,1) {$L$};
	\node[vertex] (5a) at (5,1) {};
	\node[vertex] (6a) at (6,1) {};
	\draw (4a)--(5a)--(6a);
	\node at (7,1) {$\cup$};
	\node[vertex] (8a) at (8,1) {};
	\node[vertex] (9a) at (9,1) {$L$};
	\node[vertex] (10a) at (10,1) {};
	\draw (8a)--(9a)--(10a);
	\node at (11,1) {$\cup$};
	\node[vertex] (12a) at (12,1) {};
	\node[vertex] (13a) at (13,1) {};
	\node[vertex] (14a) at (14,1) {$L$};
	\draw (12a)--(13a)--(14a);
\end{tikzpicture}}
\end{center}

Many combinatorial games, especially SP-games, have a natural tendency to break up into smaller, independent components as play progresses. For example, after several moves the empty spaces could be split into many disconnected components and a player, on their move, then has to choose a component to move in. From this, we define the \textbf{disjunctive sum} {\boldmath $G_1+G_2$} of two games $G_1$ and $G_2$, namely the game in which at each step the current player can decide to move in either game, but not both. Formally,
\begin{align*}
G_1+G_2=\{G_1^\mathcal{L}+G_2,G_1+G_2^\mathcal{L}\mid G_1^\mathcal{R}+G_2, G_1+G_2^\mathcal{R}\}.
\end{align*}

\begin{example}
This property is especially apparent in \Domineering. Consider for example a $6\times 6$ board. The position on the left below could occur during play. It is equal to the disjunctive sum of several smaller positions (empty boards) given on the right.
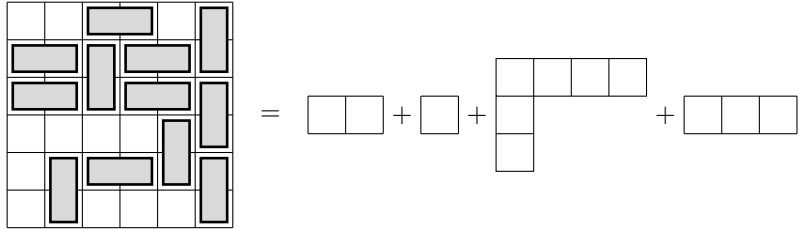
\begin{center}
\begin{tikzpicture}[scale=0.5]
	\foreach \x in {0,1,2,3,4,5,6}{
		\draw (\x,0)--(\x,6);}
	\foreach \y in {0,1,2,3,4,5,6}{
		\draw (0,\y)--(6,\y);}
	\filldraw[fill=gray!30, line width=1pt] (0.15,3.15)--(1.85,3.15)--(1.85,3.85)--(0.15,3.85)-- cycle;
	\filldraw[fill=gray!30, line width=1pt] (0.15,4.15)--(1.85,4.15)--(1.85,4.85)--(0.15,4.85)-- cycle;
	\filldraw[fill=gray!30, line width=1pt] (3.15,3.15)--(4.85,3.15)--(4.85,3.85)--(3.15,3.85)-- cycle;
	\filldraw[fill=gray!30, line width=1pt] (3.15,4.15)--(4.85,4.15)--(4.85,4.85)--(3.15,4.85)-- cycle;
	\filldraw[fill=gray!30, line width=1pt] (2.15,1.15)--(3.85,1.15)--(3.85,1.85)--(2.15,1.85)-- cycle;
	\filldraw[fill=gray!30, line width=1pt] (2.15,5.15)--(3.85,5.15)--(3.85,5.85)--(2.15,5.85)-- cycle;
	\filldraw[fill=gray!30, line width=1pt] (1.15,0.15)--(1.85,0.15)--(1.85,1.85)--(1.15,1.85)-- cycle;
	\filldraw[fill=gray!30, line width=1pt] (2.15,3.15)--(2.85,3.15)--(2.85,4.85)--(2.15,4.85)-- cycle;
	\filldraw[fill=gray!30, line width=1pt] (4.15,1.15)--(4.85,1.15)--(4.85,2.85)--(4.15,2.85)-- cycle;
	\filldraw[fill=gray!30, line width=1pt] (5.15,0.15)--(5.85,0.15)--(5.85,1.85)--(5.15,1.85)-- cycle;
	\filldraw[fill=gray!30, line width=1pt] (5.15,2.15)--(5.85,2.15)--(5.85,3.85)--(5.15,3.85)-- cycle;
	\filldraw[fill=gray!30, line width=1pt] (5.15,4.15)--(5.85,4.15)--(5.85,5.85)--(5.15,5.85)-- cycle;
	\node at (7,3) {$=$};
	\foreach \x in {8,9,10,11,12,18,19,20,21}{
		\draw (\x,2.5)--(\x,3.5);}
	\foreach \y in {2.5,3.5}{
		\draw (8,\y)--(10,\y);
		\draw (11,\y)--(12,\y);
		\draw (18,\y)--(21,\y);}
	\draw (14,4.5)--(14,1.5)--(13,1.5)--(13,4.5)--(17,4.5)--(17,3.5)--(13,3.5);
	\draw (13,2.5)--(14,2.5);
	\draw (15,3.5)--(15,4.5);
	\draw (16,3.5)--(16,4.5);
	\foreach \z in {10.5,12.5,17.5}{
		\node at (\z,3) {$+$};}
\end{tikzpicture}
\end{center}
\end{example}

We will implicitly assume that all games we consider are a summand in a disjunctive sum. Due to this, the two players do not necessarily alternate their turns in any one component (in a disjunctive sum, the players can use different components).

We can partition games into four outcome classes. The \textbf{outcome class} {\boldmath $o(G)$} of a combinatorial game $G$ indicates who will win the game when playing optimally. The outcome classes are:
\begin{itemize}
\item $\mathscr{N}$: the first (next) player can force a win;
\item $\mathscr{P}$: the second (previous) player can force a win;
\item $\mathscr{L}$: Left can force a win, no matter who plays first;
\item $\mathscr{R}$: Right can force a win, no matter who plays first.
\end{itemize}
A game whose outcome is in $\mathscr{N}$, that is one in which the first player always has a good move, is also called a first-player win. Similarly, games whose outcomes are in the other classes are called second-player win, Left win, or Right win, respectively.

We have the partial order on the outcome classes as in \cref{fig:OrderOutcomeClasses}.

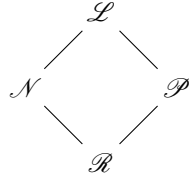
\begin{figure}[!ht]
\begin{center}
\begin{tikzpicture}
	\node (L) at (0,0) {$\mathscr{L}$};
	\node (N) at (-1,-1) {$\mathscr{N}$}
		edge (L);
	\node (P) at (1,-1) {$\mathscr{P}$}
		edge (L);
	\node (R) at (0,-2) {$\mathscr{R}$}
		edge (N)
		edge (P);
\end{tikzpicture}
\end{center}
\caption{Partial Order of Outcome Classes}
\label{fig:OrderOutcomeClasses}
\end{figure}

We say two games $G_1$ and $G_2$ are \textbf{equal}\label{def:equal} and write {\boldmath $G_1=G_2$} if $o(G_1+H)=o(G_2+H)$ for all games $H$. Game equality is an equivalence relation. The equivalence class of a game $G$ under ``$=$'' is called its \textbf{game value}. The group of all possible game values of short games under normal play with disjunctive sum as operation is denoted as $\mathbf{\G}$\label{groupG}.

For a game $G$, we say that the \textbf{negative} {\boldmath $-G$}\label{def:negative} is the game recursively defined as
	\[-G=\{-G^\mathcal{R}\mid -G^\mathcal{L}\},\]
\ie the game in which the roles of Left and Right are reversed. For example, in \domineering this is equivalent to rotating the board by $90^\circ$, while in \snort it is equivalent to switch colours of all tokens already placed. We use {\boldmath $G-H$} as shorthand for $G+(-H)$.

Similar to equality, we can also define inequalities: We say that {\boldmath $G_1>G_2$}\label{def:greater} if $o(G_1+H)> o(G_2+H)$ for all games $H$, with the partial order on the outcome classes as in \cref{fig:OrderOutcomeClasses}. Two games are incomparable, denoted {\boldmath $G_1\incomp G_2$},\label{def:incomp} if their outcome classes are incomparable, \ie if one is a first-player win and the other a second-player win. Similarly defined are $G_1\geq G_2$ and $G_1\leq G_2$.

Under normal play, we are able to determine the relationship between two games by simply determining the outcome class of their difference. Particularly, we have that $G=H$ if and only if $o(G-H)=\mathscr{P}$.

Simplicial complexes are one of the main constructions we use to study SP-games. 

\begin{definition}\label{def:simplicialComplex}
A \textbf{simplicial complex} {\boldmath $\Delta$} on a finite vertex set $V$ is a set of subsets of $V$, called \textbf{faces}, with the condition that if $A\in \Delta$ and $B\subseteq A$, then $B\in \Delta$. The \textbf{facets} of a simplicial complex $\Delta$ are the maximal faces of $\Delta$ with respect to inclusion. The \textbf{dimension of a face} is one less than the number of vertices of that face. The \textbf{dimension} {\boldmath $\dim(\Delta)$} of a simplicial complex $\Delta$ is the maximum dimension of any of its faces.
\end{definition}

Note that a simplicial complex with a fixed vertex set is uniquely determined by its facets. Thus a simplicial complex $\Delta$ with facets $F_1,\ldots, F_k$ is denoted by $\Delta=\langle F_1,\ldots, F_k\rangle$.

If a simplicial complex is of the form $\Delta=\langle \{x_1,x_2,\ldots, x_n\}\rangle$, where $\{x_1,x_2,\ldots, x_n\}$ is the vertex set of $\Delta$, we call it a \textbf{simplex}. A simplicial complex whose facets all have the same size is called \textbf{pure}.

We can assign a simplicial complex $\Delta_G$, called the \textbf{legal complex}, to each SP-game $G$ as follows. First assign variables to each basic position, using $x_1,\ldots,x_m$ for Left basic positions and $y_1,\ldots,y_n$ for Right basic positions. The vertex set of $\Delta_G$ is $\{x_1,\ldots,x_m, y_1,\ldots,y_n\}$. A subset of the vertex set forms a face if and only if the union of the corresponding basic positions forms a legal position in $G$. This construction implies that the facets of $\Delta_G$ correspond to the maximal legal positions.

\begin{example}\label{ex:Col2}
Consider \snort played on the path $P_3$. We label the spaces of the board as below.
	\begin{center}
	\begin{tikzpicture}[scale=0.5]
		\draw (0, 0) -- (3, 0) -- (3, -1) -- (0, -1) -- (0, 0);
		\draw (1, 0) -- (1, -1);
		\draw (2, 0) -- (2, -1);
		\draw (0.5, -0.5) node {1};
		\draw (1.5, -0.5) node {2};
		\draw (2.5, -0.5) node {3};
	\end{tikzpicture}
	\end{center}

The vertex set of $\Delta_{\Snort,P_3}$ is then $\{x_1,x_2,x_3,y_1,y_2,y_3\}$ and \[\Delta_{\Snort,P_3}=\langle \{x_1,y_2,x_3\}, \{y_1,x_2,y_3\}, \{x_1,y_3\}, \{x_3,y_1\}\rangle,\] 
with the listed facets corresponding to the maximal legal positions.
\end{example}

It was shown in \cite{FaridiHN2019c} that in turn given a simplicial complex $\Gamma$ with a fixed partition of the vertices into Left and Right sets, one can find an SP-game $G$ such that $\Gamma=\Delta_G$. This allows us to construct a simplicial complex instead of an SP-game directly when showing the existence of a game value. The partition of the vertex set will also be given implicitly by labels as $x_i$s or $y_j$s.

One construction of a new simplicial complex from a given one we will use is the link.
Given a simplicial complex $\Delta$ and a face $F$ of $\Delta$, the \textbf{link of \boldmath $F$ in $\Delta$}, denoted {\boldmath $\link_\Delta F$}, is defined as the subcomplex of $\Delta$ given by \[\link_\Delta F=\{G\in\Delta\mid  F\cap G=\emptyset, F\cup G\in\Delta\}.\]
In particular, if $F=\{v\}$ is a single vertex, then
\[\link_\Delta v=\{G\in\Delta\mid v\not\in G, \{v\}\cup G\in\Delta\}.\]
\begin{remark}\label{remark:optionsLink}
In SP-games, each move corresponds to a basic position, thus making a move is essentially the same as claiming a vertex in the legal complex. Furthermore, since a position is only legal if all corresponding vertices form a face, only those faces containing it are relevant from now on. This means that if we consider a game with legal complex $\Delta=\langle \{v\}\cup F_1, \{v\}\cup F_2,\ldots, \{v\}\cup F_k, F_{k+1}, \ldots, F_j\rangle$ where $F_{k+1},\ldots, F_j$ do not contain $v$, then making the move corresponding to the vertex $v$ is to a position equivalent to the game with legal complex $\Delta'=\langle F_1,F_2,\ldots, F_k\rangle$, which is the link of $v$ in $\Delta$. From here on, we will often say that a move is to $\Delta'$ when we mean the move equivalent to claiming the vertex $v$.
\end{remark}

\begin{remark}
Since the negative of a game switches Left and Right options, the legal complex of the negative of an SP-game is obtained by switching the vertices belonging to $\L$ and $\R$. Due to this, we will in this paper not demonstrate the existence of negative games, but rather assume their existence once the existence of their positive counterpart has been shown.
\end{remark}

We will also take advantage of the structure of the legal complex of a disjunctive sum of SP-games. Given two simplicial complexes $\Delta=\langle F_1, F_2,\ldots, F_i\rangle$ and $\Delta'=\langle F_1',F_2',\ldots, F_j'\rangle$ their \textbf{join} is defined as
\[\Delta *\Delta'=\langle F_1\cup F_1', F_1\cup F_2',\ldots, F_1\cup F_j',\ldots, F_i\cup F_1', F_i\cup F_2',\ldots, F_i\cup F_j'\rangle.\]
\begin{proposition}\label{thm:disjunctiveStructure}
Let $(R,B)$ and $(R',B')$ be two SP-games with legal complexes $\Delta_{R,B}$ and $\Delta_{R',B'}$. Then
\[\Delta_{(R,B)+(R',B')}=\Delta_{R,B} *\Delta_{R',B'}\]
is the legal complex of the disjunctive sum $(R,B)+(R',B')$.
\end{proposition}
\begin{proof}
A maximal legal position in the game $(R,B)+(R',B')$ is one where both the pieces placed in $(R,B)$ and the ones placed in $(R',B')$ form maximal legal positions. Thus a facet in the legal complex of $(R,B)+(R',B')$ is a union of a facet of $\Delta_{R,B}$ and a facet of $\Delta_{R',B'}$.
\end{proof}

\begin{remark}\label{remark}
As a consequence of \cref{thm:disjunctiveStructure}, if we show that two game values are taken on by SP-games, their disjunctive sum is also taken on.
\end{remark}

In combinatorial game theory, the game tree of a game is often used to study properties.
The \textbf{game tree} {\boldmath $T_G$} of a combinatorial game $G$ is a diagram constructed inductively as follows:
\begin{itemize}
\item[Step 0:] Place a vertex representing the starting position of $G$.
\item[Step $k$:] For each vertex $v$ representing a position $P$ constructed in step $k-1$ do the following: For each Left option of $P$ place a vertex $v_P$ below and to the left of $v$ and connect $v$ and $v_P$ with an edge (thus with positive slope, or oriented to the left), and similarly for all Right options.
\end{itemize}


Two games $G_1$ and $G_2$ with isomorphic game trees are called \textbf{literally equal}\label{def:isomorphic}, written as {\boldmath $G_1\cong G_2$}. 

We define $0$ to be the game $\{\emptyset\mid\emptyset\}$, so the game in which neither player has any available moves. Adding $0$ to any other games does not change it. We thus have for all games $G$, $o(G)=\mathscr{P}$ if and only if $G=0$ and we will use this fact throughout the paper to demonstrate when a game is 0.

When either of the set of options is empty, we will often leave that side of the braces empty. Thus we can also write $0=\{\,\mid\,\}$.

There are two simplifications we can use on games while still remaining in the same equivalence class. The first is to remove so-called dominated options, \ie the ones where another option is clearly preferred.
Formally, given a game $G$, a Left option $G^{L_1}$ is \textbf{dominated} by the Left option $G^{L_2}$ if $G^{L_2}\ge G^{L_1}$. Similarly, a Right option $G^{R_1}$ is dominated by the Right option $G^{R_2}$ if $G^{R_2}\le G^{R_1}$.
If for a given game $G$ the Left option $G^{L_1}$ is dominated by some other Left option, then
\[G=\{G^{\mathcal{L}}\mid G^\mathcal{R}\}=\{G^\mathcal{L}\setminus G^{L_1}\mid G^\mathcal{R}\}.\]
Similarly for Right dominated options.

For example, when playing \domineering on an L-shaped board, the game is
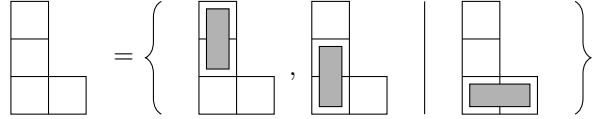
\begin{center}
\begin{tikzpicture}[scale=0.5]
\draw (0,0)--(0,3)--(1,3)--(1,1)--(2,1)--(2,0)--cycle;
\draw (0,1)--(1,1)--(1,0);
\draw (0,2)--(1,2);
\node at (3,1.5) {$=$};
\draw[decorate,decoration={brace,amplitude=6pt}] (4,0) -- (4,3);
\draw (5,0)--(5,3)--(6,3)--(6,1)--(7,1)--(7,0)--cycle;
\draw (5,1)--(6,1)--(6,0);
\draw (5,2)--(6,2);
\Ldomino{5}{1}
\node at (7.5,1) {,};
\draw (8,0)--(8,3)--(9,3)--(9,1)--(10,1)--(10,0)--cycle;
\draw (8,1)--(9,1)--(9,0);
\draw (8,2)--(9,2);
\Ldomino{8}{0}
\draw (11,0)--(11,3);
\draw (12,0)--(12,3)--(13,3)--(13,1)--(14,1)--(14,0)--cycle;
\draw (12,1)--(13,1)--(13,0);
\draw (12,2)--(13,2);
\Rdomino{12}{0}
\draw[decorate,decoration={brace,amplitude=6pt,mirror}] (15,0) -- (15,3) ;
\end{tikzpicture}
\end{center}
The Left option in the upper two squares is a Right win (only Right has a move), while the Left option in the lower two squares, having no remaining moves, is a second-player win. Thus the latter option is greater than the former, and the option in the upper two squares is dominated and can thus be removed without changing the game value. This is also apparent as Left would never make this move which gives her opponent an advantage. 

The second simplification is to replace reversible options, which are in some sense those options which have a guaranteed response. Formally, 
given a game $G$, a Left option $G^{L_1}$ is \textbf{reversible} through $G^{L_1R_1}$ if $G^{L_1R_1}\le G$. Similarly, a Right option $G^{R_1}$ is reversible through $G^{R_1L_1}$ if $G^{R_1L_1}\ge G$.
If for a given game $G$ the Left option $G^{L_1}$ is reversible through $G^{L_1R_1}$, then
\[G=\{G^\mathcal{L}\mid G^\mathcal{R}\}=\{G^\mathcal{L}\setminus G^{L_1}, G^{L_1R_1\mathcal{L}}\mid G^\mathcal{R}\}.\]
Similarly for Right options that are reversible.
Note, this is still true if $G^{L_1R_1\mathcal{L}}$ is empty.

Reversibility is unfortunately not as easy to see as domination as it requires comparing options of an option with the game itself. There will be a few cases later on where reversibility is used as a simplification. In those cases, we have used the combinatorial game theory program CGSuite \cite{CGSuite} to find the reversible options.

We say that a game is in \textbf{canonical form} if it has no dominated or reversible options. A game in canonical form in some sense is the simplest game in its equivalence class. Even more, there is only one game in canonical form in each equivalence class, so that we can talk about \emph{the} canonical form of a game.

For example, the game of \domineering on the L-shaped board above, after removing the dominated option, has the canonical form $\{0\mid\{0\mid\,\}\}$. There are no reversible options in this case.

Note that the canonical form of an SP-game is not necessarily an SP-game itself. We give an example demonstrating this in \cref{sec:ValuesFurther}.

When talking about game values, we will often represent a game value by its unique representative that is in canonical form.

As before, the game value with canonical form $\{\,\mid\,\}$ is called 0 as neither player has a move. The game $\{0\mid\,\}$ is called 1 as Left has a guaranteed move, while $\{\,\mid 0\}$ is called $-1$. The game $\{1\mid\,\}$ is then called 2 (two guaranteed moves for Left), and we can continue recursively to define all integers.

The game $\{0\mid 1\}$ is a slight advantage to Left, but not quite as much as 1 since Right actually does have a move. It turns out though that $\{0\mid 1\}+\{0\mid 1\}=1$. Thus we call this game $\displaystyle\frac{1}{2}$. Recursively, we then set $\displaystyle\{0\mid\frac{1}{2^{n-1}}\}=\frac{1}{2^n}$.

Addition of the integers and fractions as above turn out to work as in the rationals. For example $1+2=\{0\mid\,\}+\{1\mid\,\}=\{2\mid\,\}=3$.

\begin{remark}
Although all fractions can be found in combinatorial games, a short game can only take on a fraction whose denominator is a power of 2 (a dyadic rational) (see \cite[Corollary II.3.11]{Siegel2013}). The set of dyadic rationals is indicated by $\mathbb{D}$.\label{def:dyadic}
\end{remark}

Below is a list summarizing the definitions of the integers and dyadic rationals, as well as several other game values that often appear in combinatorial game theory and are therefore given shorthand notation.

\begin{itemize}
\item \textbf{Integers}: For zero we have $0=\{\,\mid\,\}$ and the other integers are recursively defined as $n=\{n-1\mid\,\}$ for $n>0$ and $n=\{\,\mid n+1\}$ for $n<0$.
\item \textbf{Fractions}: Unit fractions are recursively defined as $\displaystyle\frac{1}{2^n}=\left\{0\,\Big\vert \,\frac{1}{2^{n-1}}\right\}$. Other fractions are sums of these games.
\item \textbf{Numbers}: A game whose value is either an integer or a fraction is called a number.
\item \textbf{Switches}: A game with canonical form $\{a\mid b\}$, where $a\geq b$ are numbers, is called a switch and is written $\displaystyle \frac{a+b}{2}\pm\frac{a-b}{2}$.
\item \textbf{Nimbers}: Nimbers are recursively defined as $*1=\{0\mid 0\}$ (shorthand $*$) and $* n=\{0,*, *2,\ldots, *(n-1)\mid 0,*, *2,\ldots, *(n-1)\}$. Note that for recursive purposes we often also set $*0=0$.
\item \textbf{Up} and \textbf{down}: We have up as $\uparrow\, =\{0\mid *\}$ and down as $\downarrow \,=-\uparrow$.
\item \textbf{Tiny} and \textbf{miny}: For $G\ge 0$ a game, we have tiny-$G$ as $+_G=\{0\mid\{0\mid -G\}\}$ and miny-$G$ as $-_G=-(+_G)$.
\end{itemize}
Note that disjunctive sums of numbers, nimbers, ups, and tinies are often shortened and the `$+$' omitted. To avoid confusion between the sum $2+*$ and the nimber $*2$ for example, we will observe the order number, then up (or down), then nimbers and tinies. For example $2+*+\frac{1}{2}\,+\downarrow$ will be written as $2\frac{1}{2}\downarrow *$. If we are writing a product, such as $\uparrow+\uparrow+\uparrow$, we will use a centre dot, \ie $3\,\cdot\uparrow$.

Note that if $G=\{a\mid b\}$ where $a$ and $b$ are numbers and $a<b$, then $G$ is a number, and we have
\[G=\begin{cases}
	n & \text{if } a-b>1 \text{ and }n\text{ is the integer closest to zero such that }a<n<b;\\
	\frac{p}{2^q} & \text{if } a-b\le 1 \text{ and }q\text{ is the smallest positive integer such that }\\
	&\qquad\text{there exists a }p\text{ such that }a<\frac{p}{2^q}<b.
	\end{cases}\]
We will use this fact at times when showing that a game is a number.

\begin{example}
As examples for several of these values, we will consider \domineering positions under normal play.

(a)
\begin{center}
\begin{tikzpicture}[scale=0.5]
\draw (0,0)--(0,1)--(1,1)--(1,0)--cycle;
\node at (4,0.5) {$=\{\;\mid\;\}=0$};
\end{tikzpicture}
\end{center}

(b)
\begin{center}
\begin{tikzpicture}[scale=0.5]
\draw (0,0)--(0,2)--(1,2)--(1,0)--cycle;
\draw (0,1)--(1,1);
\node at (4,1) {$=\{0\mid\;\}=1$};
\end{tikzpicture}
\end{center}
and to get the negative we rotate the board:
\begin{center}
\begin{tikzpicture}[scale=0.5]
\draw (0,0)--(2,0)--(2,1)--(0,1)--cycle;
\draw (1,0)--(1,1);
\node at (5,0.5) {$=\{\;\mid 0\}=-1$};
\end{tikzpicture}
\end{center}

(c)
\begin{center}
\begin{tikzpicture}[scale=0.5]
\draw (0,0)--(0,2)--(1,2)--(1,0)--cycle;
\draw (0,0)--(2,0)--(2,1)--(0,1)--cycle;
\node at (5,1) {$=\{0\mid 0\}=*$};
\end{tikzpicture}
\end{center}

(d)
\begin{center}
\begin{tikzpicture}[scale=0.5]
\draw (0,0)--(0,2)--(2,2)--(2,0)--cycle;
\draw (0,1)--(2,1);
\draw (1,0)--(1,2);
\node at (3,1) {$=\left\{\right.$};
\draw (4,0)--(5,0)--(5,2)--(4,2)--cycle;
\draw (4,1)--(5,1);
\node at (6,1) {$\mid$};
\draw (7,0.5)--(9,0.5)--(9,1.5)--(7,1.5)--cycle;
\draw (8,0.5)--(8,1.5);
\node at (13,1) {$\left.\right\}=\{1\mid -1\}=\pm 1$};
\end{tikzpicture}
\end{center}

(e)
\begin{center}
\begin{tikzpicture}[scale=0.5]
\draw (0,0)--(0,3)--(1,3)--(1,0)--cycle;
\draw (0,2)--(1,2);
\draw (0,0)--(2,0)--(2,1)--(0,1)--cycle;
\node at (3,1.5) {$=\left\{\right.$};
\draw (4,1)--(6,1)--(6,2)--(4,2)--cycle;
\draw (5,1)--(5,2);
\node at (7,1.5) {$,0\mid$};
\draw (8,0.5)--(8,2.5)--(9,2.5)--(9,0.5)--cycle;
\draw (8,1.5)--(9,1.5);
\node at (12,1.5) {$\left.\right\}=\{0\mid 1\}=\frac{1}{2}$};
\end{tikzpicture}
\end{center}

Note that the Left option to $-1$ is dominated by the option to $0$, so can be ignored.

\newpage
(f)
\begin{center}
\begin{tikzpicture}[scale=0.5]
\draw (0,0)--(0,4)--(1,4)--(1,0)--cycle;
\draw (0,2)--(2,2)--(2,3)--(0,3)--cycle;
\draw (-1,1)--(1,1)--(1,2)--(-1,2)--cycle;
\node at (3,2) {$=\left\{\right.$};
\draw (4,3)--(6,3)--(6,2)--(4,2)--cycle;
\draw (5,1)--(5,3)--(6,3)--(6,1)--cycle;
\node at (7,2) {$,0\mid$};
\draw (8,3)--(10,3)--(10,2)--(8,2)--cycle;
\draw (9,1)--(9,3)--(10,3)--(10,1)--cycle;
\node at (15,2) {$\left.\right\}=\{*,0\mid *\}=\{0\mid *\}=\uparrow$};
\end{tikzpicture}
\end{center}

Note that we have only listed one of each of Left and Right's options to $*$, and that Left's option to $*$ is reversible, thus gets replaced with the Left options of 0, the empty set.
\end{example}

\begin{definition}\label{def:ValueSet}
The \textbf{value set} {\boldmath $\V$} of SP-games is the set of all possible values SP-games can exhibit under normal play, \ie all equivalence classes that contain an SP-game. It is the set $\G$ restricted to SP-games only.
\end{definition}

The question of interest is what the set $\V$ looks like. In the remainder, smaller examples have been calculated by hand. For larger examples, we have used the computer algebra program Macaulay2 \cite{M2} to construct the game in bracket and slash notation from its legal complex, which has then been put into the combinatorial game theory program CGSuite \cite{CGSuite} to obtain the canonical form.

We will show that all numbers, all nimbers, many switches, and many tinies are possible game values of SP-games, as well as that all games with small game tree are equal to some SP-game. We will at times also discuss how certain structures of the legal complex restrict possible game values. The universality of SP-games still remains open though.


\section{Small Birthdays}\label{sec:ValueBirthday}

In this section we will consider game values whose canonical forms have small game trees. For this, we will take advantage of the recursive construction of games:

The set of all short games $\widetilde{\G}$ can be defined as
\[\widetilde{\G}=\bigcup_{n\ge 0}\widetilde{\G}_n,\]
where $\widetilde{\G}_0=\{0\}$ and for $n\ge 0$
\[\widetilde{\G}_{n+1}=\{\{A\mid B\}: A,B\subseteq \widetilde{\G}_n\}.\]
If we let $\G_n$ be the set of values of elements of $\widetilde{\G}_n$, then the \textbf{birthday} {\boldmath $\birth (G)$} of a game $G$ is the least $n$ such that the game value of $G$ is in $\G_n$. Similarly, the \textbf{formal birthday} {\boldmath $\birthformal (G)$} is the least $n$ such that the literal form of $G$ is in $\widetilde{\G}_n$.

Note in particular that the birthday of a game is related to its game value, while the formal birthday relates to the first appearance of its literal form in the recursive construction of games.

The height of a game tree is the maximum number of moves from the starting position to an ending position. The elements in $\widetilde{\G}_n$ are precisely those games whose game trees have height $n$. We thus have that
the formal birthday of a game is equal to the height of its game tree.

Given this and that the height of the game tree of an SP-game $G$ is equal to the size of the largest facet in the legal complex, we have the following proposition.
\begin{proposition} Given an SP-game $G$ we have
\[\birthformal(G)=\dim(\Delta_{G})+1.\]
\end{proposition}

To illustrate the difference between the birthday and formal birthday, we will give a short example.
\begin{example}
Consider \domineering played on a $2\times 5$ grid. The maximal legal positions contain up to $5$ pieces. Thus the legal complex has dimension 4, giving that the formal birthday is $5$. This game has canonical form $\frac{1}{2}$ though, which is contained in $\G_2$, giving a birthday of $2$.
\end{example}

Motivated by the relationship between the formal birthday of an SP-game and the dimension of its legal complex we define the following sets for SP-games:
\begin{definition}\label{def:SPformalBD}
Let
\[\widetilde{\V}_n=\{G\mid G\text{ is an SP-game and }\dim(\Delta_{G})\leq n-1\},\]
where the elements of $\widetilde{\V}_n$ are in literal forms. We set $\V_n$ to be the set of values in $\widetilde{\V}_n$, \ie the set of game values of SP-games whose legal complexes have dimension at most $n-1$. 
\end{definition}

Note that we have chosen $\widetilde{\V}_n$ to be the set of SP-games with $\dim(\Delta_{G})\leq n-1$ rather than equal to $n-1$ to more closely resemble $\widetilde{\G}_n$. As empty subsets are possible in the recursive construction of $\widetilde{\G}$ the depth of the game tree does not necessarily increase with the recursion, \ie $\widetilde{\G}_n$ is the set of games with game tree depth up to $n$.

Further note that with the above notation we have \[\V=\bigcup_{n\ge 0}\V_n,\] as well as $\V_n\subseteq\G_n$.

When studying the structure of $\V$, a natural start is to consider whether for small $n$ the values born by day $n$ all come from SP-games born by day $n$, \ie if
\[\V_n=\G_n,\]
and if the set of games values of SP-games with $\dim(\Delta_{G})=n-1$ are precisely those values born on day $n$, \ie if 
\[\V_{n}\setminus\V_{n-1}=\G_n\setminus\G_{n-1}.\] We will do so for $n=0,1,2$.


\subsection{Formal Birthday 0}
Consider an SP-game $G$ with $\birthformal (G)=0$. Then the legal complex has dimension $-1$, \ie $\Delta_{G}=\emptyset$. Thus in $G$ neither Left nor Right have moves, so that $G=\{\,\mid\,\}=0$, which gives $\V_0=\{0\}$. Since $\G_0=\{0\}$, we thus have $\V_0=\G_0$.

\subsection{Formal Birthday 1}
Now consider an SP-game $G$ with formal birthday 1, so that the legal complex $\Delta_{G}$ has dimension $0$, \ie only consists of isolated vertices.
\begin{itemize}
\item If all vertices belong to $\L$, then Right has no moves while Left can move to the empty game. Thus $G=\{0\mid\,\}=1$.
\item Similarly, if all vertices belong to $\R$, then $G=\{\,\mid 0\}=-1$.
\item If both $\L$ and $\R$ are non-empty, then both Left and Right have moves to the empty game. Thus $G=\{0\mid 0\}=*$.
\end{itemize}
Since $\G_1=\{0,*,1,-1\}$, we thus have $\V_1=\V_0\cup\{*,1,-1\}=\{0,*,1,-1\}=\G_1$ and $\V_1\setminus\V_0=\G_1\setminus\G_0$.

\subsection{Formal Birthday 2}\label{sec:valueBirthday2}

First note that 
\begin{align*}
\G_2&=\Big\{0,*,*2,\pm 1,\uparrow, \downarrow, \uparrow *, \downarrow *, \{1\mid 0,*\}, \{0,*\mid -1\},\frac{1}{2}, -\frac{1}{2}, \{1\mid *\}, \{*\mid -1\},\\
&\qquad\qquad\frac{1}{2}\pm\frac{1}{2}, -\frac{1}{2}\pm\frac{1}{2},1, -1,1*, -1*,2, -2\Big\}
\end{align*}

Now consider an SP-game $G$ whose legal complex has dimension 1, \ie it is a graph. Thus $G$ is born by day $2$. Note that as previously mentioned as soon as we have shown a positive value exists, we assume to have shown the existence of the negative as well (through switching the bipartition).
\begin{itemize}
\item If all vertices belong to $\L$, then Left can move to a single vertex belonging to $\L$, \ie to the game $1$. Thus $G=\{1\mid\,\}=2$.
\item If $\Delta_{G}=\langle \{x_1,y_1\}\rangle$, then Left can move to $\langle\{y_1\}\rangle$, \ie $-1$. Similarly for Right, thus $G=\{-1\mid 1\}=0$.
\item If $\Delta_{G}=\langle \{x_1,x_2\}, \{x_1,y_1\}\rangle$, then $G=\{1,*\mid 1\}$. The Left option to $*$ is reversible and gets replaced with the empty set, thus $G=\{1\mid 1\}=1*$.
\item If $\Delta_{G}=\langle \{x_1,x_2\}, \{y_1,y_2\}\rangle$, then $G=\{1\mid -1\}=\pm 1$.
\item If $\Delta_{G}=\langle \{x_1,x_2\}, \{x_2,y_1\}, \{y_1,y_2\}, \{y_2,x_3\}\rangle$, then $G=\{1, *, -1\mid *\}=\{1\mid *\}$.
\item If $\Delta_G=\langle \{x_1,x_2\}, \{y_1\}\rangle$, then $\displaystyle G=\{1\mid 0\}=\frac{1}{2}\pm\frac{1}{2}$.
\item  If $\Delta_{G}=\langle \{x_1,y_1\}, \{x_2\}\rangle$, then $\displaystyle G=\{-1,0\mid 1\}=\{0\mid 1\}=\frac{1}{2}$.
\item  If $\Delta_{G}=\langle \{x_1,y_1\}, \{x_2\}, \{y_2\}\rangle$, then $G=\{-1,0\mid 1,0\}=\{0\mid 0\}=*$.
\item If $\Delta_{G}=\langle \{x_1,y_1\}, \{y_1,y_2\}, \{y_2,x_2\}, \{x_2,x_3\}, \{x_3,y_3\}, \{x_4\}, \{y_4\}\rangle$, then we have ${G=\{-1,*,0\mid *,1,0\}=\{0,* \mid 0,*\}=*2}$.
\item If $\Delta_{G}=\langle \{x_1,y_1\},\{y_1,y_2\},\{y_2,x_2\}, \{x_2,x_3\}, \{x_3,y_3\}, \{x_4\}\rangle$, then\\
$G={\{-1,*,0\mid *,1\}}=\{0\mid *\}=\uparrow$.
\item If $\Delta_{G}=\langle \{x_1,x_2\}, \{x_2,y_1\}, \{y_1,y_2\}, \{y_2,x_3\}, \{y_3\}\rangle$, then $G=\{1,*,-1 \mid *,0\}=\{1 \mid 0,*\}$.
\item If $\Delta_{G}=\langle \{x_1,y_1\}, \{x_1,x_2\}, \{x_2,y_1\}, \{y_2\}, \{x_3\}\rangle$, then $G=\{*,0\mid 1,0\}=\{0,* \mid 0\}=\uparrow *$.
\end{itemize}

The values 1 and $-1$ are not possible if the legal complex has dimension 1 (see \cref{thm:ValueIntDimImpossible}). As these are elements of $\V_1$ though, we have shown that $\V_2=\G_2$. We also have $\V_2\setminus V_1=\G_2\setminus \G_1$.

\vspace{1em}

Although it seems reasonable to next look at whether all values of other birthdays are possible, the size of $\G_3$ alone is 1474.
We will thus turn to more general existence results independent of the birthday.

\section{Specific Values of Partizan SP-Games}\label{sec:values}
\subsection{Integers}
We will begin by showing that all positive integers (and thus also the negatives) are possible values of SP-games.
\begin{proposition}\label{thm:valueInteger}
Let $G$ be an SP-game with legal complex the simplex $\Delta_{G}=\langle \{x_1,\ldots,x_n\}\rangle$ with $n\ge 0$. Then $G=n$.
\end{proposition}
\begin{proof}
We will prove this by induction on $n$.

If $n=0$, then $\Delta_{G}=\langle\emptyset\rangle$. We have shown previously that $G=0$ in this case.

Now assume without loss of generality that the SP-game with legal complex $\langle \{x_1,\ldots,x_{n-1}\}\rangle$ has value $n-1$.

If $\Delta_{G}=\langle \{x_1,\ldots,x_n\}\rangle$, then Right has no moves, while Left, without loss of generality, can move to $\langle \{x_1,\ldots,x_{n-1}\}\rangle$. By the induction hypothesis, we then have $G=\{n-1\mid\,\}=n$.
\end{proof}

From this and our knowledge about disjunctive sums, we get an immediate corollary on the value of a game whose legal complex is a simplex.
\begin{corollary}\label{thm:valueSimplex}
Let $G$ be an SP-game such that $\Delta_{G}$ is the simplex \[\langle\{x_1,\ldots,x_m,y_1,\ldots,y_n\}\rangle.\] Then $G$ has value $m-n$.
\end{corollary}
\begin{proof}
We can write $\Delta_{G}$ as a join:
\begin{align*}
	\Delta_{G}&=\langle\{x_1,\ldots,x_m,y_1,\ldots,y_n\}\rangle\\
			&=\langle\{x_1,\ldots,x_m\}\rangle * \langle\{y_1,\ldots,y_n\}\rangle.
\end{align*}
If we let $G'$ and $G''$ be SP-games such that $\Delta_{G'}=\langle\{x_1,\ldots,x_m\}\rangle$ and $\Delta_{G''}=\langle\{y_1,\ldots,y_n\}\rangle$, then by \cref{thm:disjunctiveStructure} we have $G=G'+G''$. By our previous result we further have $G'=m$ and $G''=-n$, so that $G=m-n$ as desired.
\end{proof}

With these two results on the existence of integers, we now turn to looking at the existence of integers in specific dimensions, thus checking if $n$ is a value of a game in $\widetilde{\V}_k\setminus\widetilde{\V}_{k-1}$ where we let $k$ vary.

Note that since $\birthformal (G)=\dim \Delta_{G}+1$, and the integer $n$ has birthday $|n|$, we cannot get $n$ as a value at dimension less than $|n|-1$. We have also already shown above that $n$ is possible at dimension equal to $|n|-1$. We will show below that in dimension $|n|$ the integer $n$ is not possible, while for all larger dimensions it is.

\begin{proposition}\label{thm:ValueIntDimImpossible}
An SP-game $G$ with $\dim \Delta_{G}=n$ cannot take on the value $n$ (or $-n$) under normal play.
\end{proposition}
\begin{proof}
We will show by induction that the value $n$ is not possible. That $-n$ is not possible follows immediately since it is the negative, \ie could be achieved by switching the bipartition of vertices.

\textit{Base case}: As shown in \cref{sec:ValueBirthday}, we cannot get 0 with dimension 0.

\textit{Induction hypothesis}: Assume that an SP-game with legal complex of dimension $n-1$ cannot take on value $n-1$.

\textit{Induction step}: Assume that $G$ has the value $n$, \ie $G=\{n-1|\,\}$. Since $G$ is born by day $n+1$ (since $\dim\Delta_G=n$), we have that all Left options of $G$ have to be born by day $n$. Thus the Left option to $n-1$ in the canonical form of $G$ cannot have come through reversing (reversing an option born by day $k$ results in options born by day $k-2$). Thus in $\Delta$ there exists a facet $F$ of dimension $n$ such that $F=\{x_i\}\cup F'$ where the game equivalent to $\langle F'\rangle$ has value $n-1$. But $\langle F'\rangle$ has dimension $n-1$, a contradiction to the induction hypothesis.
\end{proof}

\begin{proposition}
Let \[U=\bigl\{\{x_1,x_2,\ldots, x_{k+1}\}\bigr\}\cup\bigl\{\{x_{i_0},x_{i_1},\ldots, x_{i_n},y\}\mid 1\le i_0,\ldots, i_n\le k+1\bigr\},\] where $k\ge n+1$, and let the facets of $\Delta_{G}$ be the elements of $U$. Then $\Delta_{G}$ has dimension $k$ and $G$ has value $n$. 
\end{proposition}
\begin{proof}
Since $k\ge n+1$, we have that $k+1\ge n+1+1$, \ie there are at least as many elements in $\{x_1,\ldots, x_{k+1}\}$ as in $\{x_{i_0},x_{i_1},\ldots, x_{i_n},y\}$. Thus $\{x_1,\ldots, x_{k+1}\}$ is a facet of maximal dimension, which shows that $\dim(\Delta_{G})=k$.

In $G$, Right's only move is to $\langle\{x_{i_0},x_{i_1},\ldots, x_{i_n}\}\mid 1\le i_0,\ldots, i_n\le k+1\rangle$, which has value $n+1$. Left's moves are symmetric, so assume without loss of generality she moves in $x_{k+1}$. This is then to $\Delta'$ which has the facets $\{x_1,x_2,\ldots, x_{k}\}$ and $\{x_{i_0},x_{i_1},\ldots, x_{i_{n-1}},y\}$ where $1\le i_0,\ldots, i_n\le k$. By induction, it can now be easily seen that $G$ has value
\[G=\{\ldots\{\{k-1-n|0\}|1\}\ldots|n+1\}.\]

To prove that $G$ has value $n$, we will use induction on $n$.

\textit{Base case:} If $n=0$, then $\Delta_{G}=\langle \{x_1,\ldots, x_{k+1}\}, \{x_1,y\},\ldots, \{x_{k+1},y\}\rangle$. Thus $G=\{\{k-1|0\}|1\}$ which is $0$ for all $k\ge 1$ since it is a second player win.

\textit{Induction hypothesis:} Assume that for a fixed $j$ and for all $k\geq j+1$, we have $\{\ldots\{\{k-1-j|0\}|1\}\ldots|j+1\}=j$.

\textit{Induction step:} Suppose that $G=\{\{\ldots\{\{k-1-(j+1)|0\}|1\}\ldots|j+1\}|(j+1)+1\}$ for $k\ge j+2$. By the induction hypothesis we have that $\{\ldots\{\{(k-1)-1-j|0\}|1\}\ldots|j+1\}=j$ since $k-1\ge j+1$. Thus $G=\{j|j+2\}=j+1$.
\end{proof}

\begin{example}
We will construct a simplicial complex $\Delta$ of dimension 2 that is the legal complex of an SP-game $G$ with value 1. In this case $U=\{\{x_1,x_2,x_3\}\}\cup\{\{x_1,x_2,y\}, \{x_1,x_3,y\}, \{x_2,x_3,y\}\}$ and the facets of $\Delta$ are the sets of $U$. 

Right's only move is on $y$, after which Left has two remaining moves, \ie this option has value 2.

All of Left's options are symmetric, so we will assume without loss of generality that she moves on $x_1$.
\begin{itemize}
\item After Right moved on $y$, Left still has one move, thus this Right option has value 1.
\item After Left moved on $x_2$, both Left and Right have options to 0. Thus this Left option has value $\{0\mid 0\}=*$.
\end{itemize}
Thus the Left option of $x_1$ has value $\{*\mid 1\}$.

In total, $G$ has value $\{\{*\mid 1\}\mid 2\}=1$.
\end{example}

To summarize, the integer $n$ is a possible value of an SP-game if the legal complex has dimension $n-1$ or greater than $n$, but not dimensions $n$ or less than $n-1$. 

\subsection{Fractions}
The following construction shows that all fractions of the form $\frac{1}{2^n}$ are possible values of SP-games. This construction is also minimal in the sense that the dimension of the legal complex is one lower than the birthday of the fraction. All other fractions can be obtained through disjunctive sums.

\begin{theorem}\label{thm:valueFractions}
Let $S_1, S_2,\ldots, S_{2^n}$ be the subsets of $\{y_1, y_2,\ldots, y_n\}$. Let
\[\Delta_{G}=\langle \{x_1\}\cup S_1, \{x_2\}\cup S_2,\ldots, \{x_{2^n}\}\cup S_{2^n}\rangle.\]
Then $G$ has value $\displaystyle\frac{1}{2^n}$.
\end{theorem}
\begin{proof}
We will prove this by induction on $n$.

\textit{Base case:} 
For $n=0$ we have $\Delta_{G}=\langle\{x_1\}\rangle$. We have shown in the previous section that $G=1$ in that case.

\textit{Induction hypothesis:} Let $S_1', S_2',\ldots, S_{2^{n-1}}'$ be the subsets of $\{y_1,y_2,\ldots, y_{n-1}\}$. Assume that the game with legal complex $\displaystyle\Delta'=\langle \{x_1'\}\cup S_1',\ldots, \{x_{2^{n-1}}'\}\cup S_{2^{n-1}}'\rangle$ has value $\displaystyle\frac{1}{2^{n-1}}$.

\textit{Induction step:} Without loss of generality, assume that $S_1,S_2,\ldots, S_{2^n}$ are ordered such that $S_{2^n}=\emptyset$ and the sets $S_1, S_2,\ldots, S_{2^{n-1}}$ are those containing $y_n$.

Left has the options to move to the games with legal complexes $\langle S_1\rangle$, $\langle S_2\rangle$, $\ldots$, $\langle S_{2^n}\rangle$. All of those options, except for the one corresponding to $\langle S_{2^n}\rangle=\langle\emptyset\rangle$, will be negative. The option corresponding to $\langle\emptyset\rangle$ is $0$, and thus dominates all other options.

All of Right's moves are symmetric. We will assume without loss of generality that he moves in $y_n$. This option leaves us with the game with legal complex $\langle \{x_1\}\cup S_1\setminus\{y_n\}, \{x_2\}\cup S_2\setminus\{y_n\},\ldots, \{x_{2^{n-1}}\}\cup S_{2^{n-1}}\setminus\{y_n\}\rangle$. This game has value $\displaystyle\frac{1}{2^{n-1}}$ by the induction hypothesis.

Thus $\displaystyle G=\left\{0\,\Big\vert \,\frac{1}{2^{n-1}}\right\}=\frac{1}{2^n}$.
\end{proof}

The following is then an immediate consequence using disjunctive sums.
\begin{corollary}\label{thm:valueAllFractions}
Given any dyadic rational $\displaystyle\frac{a}{2^n}$ there exists an SP-game $G$ such that $\displaystyle G=\frac{a}{2^n}$.
\end{corollary}

\subsection{Switches}
We will show that all switches $\{a\mid b\}$ with $a\geq b$ being integers are possible as game values of SP-games. We will begin with $a$ non-negative and $b$ non-positive.

\begin{proposition}\label{thm:valueSwitchA}

If $a,b\ge 0$ are integers, then the SP-game $G$ with $\Delta_{G}=\langle \{x_0,\ldots, x_a\}, \{y_0,\ldots, y_b\}\rangle$ has value $\{a\mid -b\}$.
\end{proposition}
\begin{proof}
Left's moves are all to a simplex consisting of $a$ Left vertices, thus has value $a$. Similarly Right going first will move to $-b$. Thus $G$ has value $\{a\mid -b\}=\frac{a-b}{2}\pm\frac{a+b}{2}$.
\end{proof}

If a connected legal complex is desired and $a,b\ge 1$, then we can also add in the face $\{x_0,y_0\}$, and a move in this face will be dominated, thus giving the same value.

Next we consider the case in which $a$ is positive and $b$ non-negative. The case of $0\ge a>b$ is the negative of this.

\begin{proposition}\label{thm:valueSwitchB}
If $a>b\ge 0$ are integers, then the SP-game $G$ with $\Delta_{G}=\langle \{x_1,\ldots, x_{a+1}\}, \{x_1,\ldots, x_b,y\}\rangle$ has value $\{a\mid b\}$.
\end{proposition}
\begin{proof}

We will prove this by induction on $a$.

If $a=1$, then necessarily $b=0$, so that $\Delta_{G}=\langle\{x_1,x_2\},\{y\}\rangle$, which we have shown in the previous result has value $\{1\mid 0\}$

Now assume that if $a>k>0$, $j\ge 0$ and $\Delta_{G'}=\langle \{x_1,\ldots, x_{k+1}\}, \{x_1,\ldots, x_j,y\}\rangle$ then $G'=\{k\mid j\}$.

If Left moves in any of $x_1,\ldots,x_b$, say without loss of generality in $x_1$, this is to $\langle\{x_2,\ldots,x_a\},\{x_2,\ldots,x_b,y\}\rangle$. By induction, this has value $\{a-1\mid b-1\}$.

If Left moves in any of $x_{b+1},\ldots, x_a$, say without loss of generality in $x_a$, then it is to $\langle\{x_0,\ldots,x_{a-1}\}\rangle$. This has value $a$.

Right's only move is to $\langle\{x_1,\ldots,x_b\}\rangle$, which has value $b$. Thus $G$ has value $\{\{a-1\mid b-1\},a\mid b\}=\{a\mid b\}$.
\end{proof}

Note that the above can also be shown by using that disjunctive sum of games corresponds to the join of simplicial complexes (see \cref{thm:disjunctiveStructure}). The simplicial complex can be written as the following join:
\[\langle \{x_1,\ldots, x_{a+1}\}, \{x_1,\ldots, x_b,y\}\rangle=\langle \{x_1,\ldots, x_{b}\}\rangle*\langle \{x_{b+1},\ldots, x_{a+1}\}, \{y\}\rangle\]
The first gives a game with value $b$ by \cref{thm:valueInteger}, and the second a game with value $\{a-b\mid 0\}$ by \cref{thm:valueSwitchA}. And we indeed have
\[b+\{a-b\mid 0\}=\{a\mid b\}.\]

\subsection{Tiny and Miny}

We will show that all $+_n$, where $n$ is a positive integer, are possible game values of SP-games. Since $+_0=\,\uparrow$, we have already shown the existence of this value (see \cref{sec:valueBirthday2}).

\begin{proposition}\label{thm:valueTiny}
If $n$ is a positive integer, then the SP-game $G$ with $\Delta_{G}=\langle \{y_1,\ldots, y_{n+1}\}, \{x_1,y_1\},\ldots, \{x_1,y_{n+1}\}, \{x_2\}\rangle$ has value $+_n$.
\end{proposition}
\begin{proof}
Left's move in $x_1$ is to $\langle\{y_1\},\ldots,\{y_{n+1}\}\rangle$, which has value $-1$. The move in $x_2$ is to $0$, and this also dominates the move to $-1$.

Right's moves are all symmetric, so assume without loss of generality that he makes the move corresponding to $y_1$. Then this is to $\langle\{y_2,\ldots,y_{n+1}\},\{x_1\}\rangle$, and it can be easily seen that this has value $\{0\mid -n\}$.

Thus $G$ has value $\{0\mid \{0\mid -n\}\}=+_n$.
\end{proof}

We know from \cite{UiterwijkB2015} that several other tinies are also possible values of \Domineering, thus are elements of $\V$, for example $+_{1/2}$, $+_{1/4}$, and $+_{(1/2)*}$.

\subsection{Nimbers}\label{sec:nimbers}
Contrary to other values, we will show the existence of nimbers as game values of SP-games by constructing the ruleset and board directly, rather than through the legal complex.

\begin{proposition}\label{thm:valueNimber}
For every nimber $*n$ there exists an SP-game $(R,B)$ that has value $*n$.
\end{proposition}
\begin{proof}
Let $R$ be the ruleset in which both Left and Right have as their pieces $K_1,K_2,\ldots, K_n$, played on $B=K_n$. We will show by induction that this has value $*n$. Note that $*0=0$ and $*1=*$.

\textit{Base case:} For $n=0$, the board is empty, and Left and Right have no options, thus $(R,B)=0$.

\textit{Induction hypothesis:} Assume that for all $j<n$, the game in which Left and Right can play pieces $K_1,\ldots, K_j$ on the board $K_j$ has value $*j$.

\textit{Induction step:} We now have as our board $B=K_n$ and pieces $K_1,\ldots, K_n$. Suppose that either player places $K_l$ as their first piece. The game from this point is now equivalent to playing $K_1,\ldots, K_n$ on $K_{n-l}$. Since the pieces $K_{l+1},\ldots,K_n$ cannot be placed, this has value $*l$ by induction hypothesis. Thus we have $(R,B)=\{0,*,\ldots, *(n-1)\mid 0,*,\ldots, *(n-1)\}=*n$.
\end{proof}

\begin{remark}
Note that each game defined above is equivalent to the game \textsc{Nim} on a single pile. Playing \textsc{Nim} on several piles is equivalent to playing a disjunctive sum of this game, showing that \textsc{Nim} can be thought of as an SP-game.
\end{remark}

\section{Impartial SP-Games}\label{sec:impartial}

The game we constructed in \cref{thm:valueNimber} to show the existence of the nimbers is an impartial game. \textbf{Impartial games} are those in which the two players always have the same options and thus we do not distinguish Left from Right.

In particular, this means that making a distinction between basic Left positions and basic Right positions does not give us any additional information. We can instead introduce a simplified version of the legal complex for impartial games.

\begin{definition}\label{def:impartialGameComplex}
For an impartial SP-game $G$ the \textbf{impartial legal complex} {\boldmath $\Delta^\I_{R,B}$} is the simplicial complex with vertex set $\{x_1,\ldots, x_n\}$ and whose faces consist of vertices corresponding to the basic positions forming a legal position. 
\end{definition}

Just as for the general case, each simplicial complex is the impartial legal complex of some impartial SP-game, and the disjunctive sum of two games corresponds to the join of the impartial legal complexes.

For an impartial combinatorial game, since there is no differentiation between Left and Right, game options are just listed in curly brackets, without a divider. For example, the value $*$ can be written as $\{0\}$. A general option of the position $P$ is indicated as $P'$, so $P=\{P_1', P_2',\ldots, P_k'\}$.

The well-known Sprague-Grundy Theorem states that each impartial combinatorial game has at its value a nimber. To find which nimber an impartial game $G$ is equal to, we use the \textbf{minimal excluded value} (mex). The {\boldmath $\mex(A)$} of a finite set of non-negative integers $A$ is the least non-negative integer not contained in $A$. If the impartial game $G$ inductively is equal to $\{*n_1,\ldots, *n_k\}$, then $G=*n$ where $n=\mex\{n_1,\ldots,n_k\}$.

By \cref{thm:valueNimber} we already know that all nimbers, thus all possible game values, are game values of some impartial SP-games. We are able to show though how some properties of the impartial legal complex restrict which nimbers are possible.

Using that a game has value 0 if and only if it is a second-player win, \ie the second player always has a good responding move, we have the following two results.
\begin{proposition}\label{thm:ImpartialEven}
Let $G$ be an impartial SP-game. If all facets of $\Delta_{G}^\I$ have odd dimension, thus even size, then $G$ has value 0.
\end{proposition}
\begin{proof}
Since all facets of $\Delta_{G}^\I$ have even size, all maximal legal positions of $G$ also have an even number of pieces. Thus whenever the game ends, the second player will have made the last move, thus wins.
\end{proof}
\begin{proposition}
If an impartial SP-game $G$ has value $0$, then the impartial legal complex has at least one facet of odd dimension, thus even size.
\end{proposition}
\begin{proof}
An impartial SP-game can only be a second-player win if there is at least one maximal legal position with an even number of moves, thus there is a facet of even size.
\end{proof}

We are also able to determine the value of $G$ immediately if its impartial legal complex is pure.
\begin{proposition}\label{thm:pureValue}
If $\Delta_{G}^\I$ is a pure $(n-1)$-dimensional simplicial complex, then the corresponding impartial SP-game $G$ has value $0$ if $n$ is even and $*$ if $n$ is odd.
\end{proposition}
\begin{proof}
The case in which $n$ is even was already proven in \cref{thm:ImpartialEven}.

We will now consider the case in which $n$ is odd. In this case, any move will be to a pure $(n-2)$-dimensional simplicial complex $\Delta'$. We already know that the impartial SP-game corresponding to $\Delta'$ has value $0$. Thus $G=\{0\}=*$.
\end{proof}

The above result should be particularly useful in applications since proving the impartial complex is pure immediately results in restricting the possible game values to $0$ and $*$.

We also have
\begin{corollary}
If $\Delta$ is the disjoint union of pure simplicial complexes $\Delta_1,\ldots,\Delta_k$ with dimensions $d_1,\ldots,d_k$, then the corresponding impartial SP-game $G$ has value $0$ if all $d_i$ are odd, $*$ if all $d_i$ are even, and $*2$ otherwise.
\end{corollary}
\begin{proof}
In this case, a move always forces the game into a single component. If all $d_i$ are odd, then any move is in a component with value $0$, which is thus the value of the entire complex. Similarly if all $d_i$ are even. If there is a mix though, there are moves to $0$ (moving in an even-dimensional complex to an odd dimensional one, all pure) and to $*$  (moving in odd dimensional complex), thus the value is $*2$.
\end{proof}

\section{Further Work}\label{sec:ValuesFurther}
The study of whether specific game values are elements of $\V$ can be continued by looking at values such as switches $\{a\mid b\}$ where $a$ and/or $b$ are not integers, tiny $+_G$ where $G$ is a fraction or non-number, or other values we have not yet discussed at all, for example $\uparrow^n$ and $\uparrow^{[n]}$.

Ideally, we would like to have a recursive construction that works similarly to $\G_n$. This seems difficult though as simply combining two simplicial complexes, such as joining at a vertex or a face, often creates unwanted options besides the ones needed.

In case that $\V$ is not equal to $\G$, the question of course is which values are not possible. It is generally more difficult to show non-existence of a value though, and here again the legal complexes should be of value.

Related to this, we can also ask when the canonical form of an equivalence class containing an SP-game is itself literally equal to an SP-game. Using that the order of moves does not matter, we know that this is not always the case. Consider the following example:
\begin{example}
We have shown that $-\frac{1}{2}$ is the game value of some SP-game. Now consider the canonical form of $-\frac{1}{2}$, which is $\{\{\;\mid 0\}\mid 0\}$. The game tree of this canonical form is
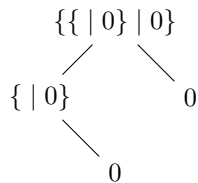
\begin{center}
\begin{tikzpicture}
	\node (a) at (0,0) {$\{\{\;\mid 0\}\mid 0\}$};
	\node (b) at (-1,-1) {$\{\;\mid 0\}$} edge (a);
	\node (c) at (1,-1) {$0$} edge (a);
	\node (d) at (0,-2) {$0$} edge (b);
\end{tikzpicture}
\end{center}
Since there exists a Left move followed by a Right, but no Right followed by a Left, this is not the game tree of an SP-game.
\end{example}


Similar to questions in the general case, we are interested in whether we are able to completely describe how the structure of the impartial legal complex determines the normal play value of the corresponding impartial game.

In more general terms, one can also ask which values are possible under mis\`ere play, and all related questions above.

\bibliographystyle{plain}
\bibliography{Biblio2019Aug}
\end{document}